%% file: main.tex
\documentclass[11 pt,reqno]{amsart}

\usepackage[utf8]{inputenc}
\input{data}

\input{packages}

\input{commands}

\usepackage{todonotes}

\begin{document}

\begin{abstract}
    Through the means of an alternative and less algebraic method, an explicit expression for the isometry groups of the six-dimensional homogeneous nearly Kähler manifolds is provided. 
    %An explicit expression for the isometry groups of the six-dimensional homogeneous nearly Kähler manifolds is provided. 
    %We present an alternative way to obtain the full isometry group of a homogeneous manifold.
\end{abstract}

\maketitle
\section{Introduction}

A connected homogeneous Riemannian manifold $M$ is a Riemannian manifold such that its isometry group $G$ acts transitively on it. 
Being an orbit, it is immediately diffeomorphic to the quotient $G/H$ where $p$ is any point in $M$ and $H$ is the isotropy subgroup associated to $p$.
By connectivity, we may write $M=G_o/H_o$ where $G_o$ and $H_o$ are the connected components of the identity. 
Because of this reduction, sometimes it is uncertain what the full isometry group is of a manifold $M$. 
When it comes to the study of submanifolds, knowing the full isometry group of the ambient space proves to be particularly useful, since it can help us to either distinguish or identify certain immersions.
To give examples of this, we recall that a strict nearly Kähler manifold is an almost Hermitian manifold $(M,g,J)$ such that the covariant derivative of $J$ is skew symmetric and non-degenerate. 
Then, in \cite{wang} the totally geodesic Lagrangian submanifolds of $\Ss^3\times\Ss^3$ were classified into six different immersions, but, with further research, the authors will later realize there were only two. 
Similar situations happened in \cite{zekesomehyper}. 
In \cite{kamil}, three congruent immersions in the nearly Kähler flag manifold were identified, but the isometry mapping one into the other was not explicitly given. 
Clearly, the isometry groups or their actions are not sufficiently known to detect these redundant submanifolds in the classification results of the above examples at the time. 
Also note that all of these examples occurred in the study of submanifolds of the homogeneous six-dimensional nearly Kähler spaces.

The homogeneous six-dimensional nearly Kähler spaces are examples of 3-symmetric spaces. Recall that a 3-symmetric space is a homogeneous manifold $G/H$, where $G$ has an automorphism $\theta$ of order three such that $G^\theta_o\subset H\subset G^\theta$, with $G^\theta=\{g\in G:\theta(g)=g\}$ and $G^\theta_o$ the connected component of the identity in $G^\theta$.
From this definition, we deduce a natural almost complex structure on $G/H$ given by $J=\frac{2}{\sqrt{3}}(\theta_*+\frac{1}{2}\id)$. In \cite{gray3symmetries}, Gray showed that $G/H$ admits a Riemannian metric $g$ such that $(G/H, g,J)$ is a homogeneous nearly Kähler manifold. 
Gray and Wolf conjectured that the converse was also true, later proven in \cite{butruille} by Butruille, where he also classified all six-dimensional homogeneous nearly Kähler manifolds: $\Ss^6$, $\Ss^3\times\Ss^3$, $\CP[3]$, and the flag manifold $F(\C^3)$.

The purpose of this article is to survey and gather some results related to isometries of the six-dimensional homogeneous nearly Kähler manifolds.
Moreover, we fill in what is not known to date and present it in an explicit and understandable way for new people in the field.
We give an expression of the full isometry group of all six-dimensional homogeneous nearly Kähler manifolds, and make clear how they act. 
We also provide a different proof than the one in \cite{shankar2001} for the isometry group of $\CP$. 
The full isometry group of $\Ss^3\times\Ss^3$ was first given in \cite{Vasquez2016}. 
In this survey, we provide a proof for the isometry group $\Ss^3\times\Ss^3$ similar to the one for $\Sl\times\Sl$ in~\cite{anarellahomogeneo}, which can be also adapted to other homogeneous nearly Kähler spaces, including the pseudo-Riemannian analogues.
% We also note that a proof of the isometry group of $\Ss^3\times\Ss^3$ is given in \cite{Vasquez2016}, but that the method presented here instead works for all six-dimensional homogeneous nearly Kähler manifolds and also for the pseudo-Riemannian analogues. 
In particular, the method also works for the nearly Kähler flag manifold, whose full isometry group is not known in the literature at the moment.

The main theorem we prove is the following.
\begin{theorem}
    The six-dimensional homogeneous nearly Kähler manifolds have the following isometry groups:
    \begin{itemize}
        \item $\iso(\Ss^6)=\Ort(7)$
        \item $\iso(\Ss^3\times\Ss^3)=\mathrm{P}(\SU(2)\times\SU(2)\times\SU(2))\rtimes S_3$
        \item $\iso(\C P^3)= \PSp(2)\rtimes \Z_2$
        \item $\iso(F(\C^3))=\PSU(3)\times S_3\rtimes \Z_2$
    \end{itemize}
where $\mathrm{P}(\SU(2)\times\SU(2)\times\SU(2))$ denotes the group $(\SU(2)\times\SU(2)\times\SU(2))/\Z_2$, $S_3$ is the symmetric group of order six, $\Sp(2) = \Sp(4, \C) \cap \U(4)$ and $\PSp(2) = \Sp(2)/\Z_2$, $\PSU(3)=\SU(3)/\Z_3$ and $\Z_n=\Z/n\Z$.
\end{theorem}
% a projective group PG is the group G quotiented by its center.

Of course, the isometry group of the round sphere is known. To prove the result for the other spaces, we will work as follows. First, we consider the curvature tensor of the space and from this prove that there are certain geometric structures on the manifold that are isometry invariants. These geometric invariants than lead to certain distributions being respected by every isometry. Using isometries of the proposed groups, we can then transform every isometry to an isometry of the proposed group and we are done. In the coming sections we recall constructions of and apply this method to the three remaining six-dimensional homogeneous nearly Kähler spaces. 

\section{The nearly Kähler \texorpdfstring{$\Ss^3\times\Ss^3$}{S³xS³}}
The main source for information about $\Ss^3\times\Ss^3$ is \cite{propertiesofs3s3}.

Let $\Ss^3$ be the three-sphere. 
Thinking of it as the set of unit quaternions, it inherits a Lie group structure. 
The tangent space at a point $p$ can be described as the set $T_p\Ss^3=\{p\alpha:\alpha\in \mathrm{Im}(\mathbb{H})\}$. 
It has a natural Riemannian metric $\li,\ri$, given by $\li p\alpha, p\beta\ri=\frac{1}{2}(\alpha\beta+\beta\alpha)$.

The homogeneous nearly Kähler $\Ss^3\times\Ss^3$ carries the metric $g$ and the almost complex structure $J$ given by
\[
g((p\alpha,q\beta),(p\gamma,q\delta))=\frac{4}{3}(\li \alpha ,\gamma\ri +\li \beta,\delta\ri)-\frac{2}{3}(\li\beta ,\gamma\ri+\li\alpha,\delta\ri)
\]
and
\[
J(p\alpha,q\beta)=\frac{1}{\sqrt{3}}(p(-\alpha+2\beta),q(-2\alpha+\beta)),
\]
respectively. The nearly Kähler $\Ss^3\times\Ss^3$ also carries an additional almost product structure $P$ given by
\[
P(p\alpha,q\beta)=(p\beta,q\alpha).
\]
With these structures, the curvature tensor of $\Ss^3\times\Ss^3$ can be expressed as
\begin{equation}
    \begin{split}
        R(U, V )W &=\frac{5}{12}\Big(g(V, W)U-g(U, W)V\Big)\\
        &\quad+\frac{1}{12}\Big(g(JV, W)JU-g(JU, W)JV-2g(JU, V )JW\Big)\\
        &\quad+\frac{1}{3}\Big(g(P V, W)P U-g(P U, W)P V\\
        &\quad+ g(JP V, W)JP U-g(JP U, W)JPV\Big).
    \end{split}\label{curvtarures3s3}
\end{equation} 

\begin{lemma}\label{proppe}
    The following relations between $g$ $P$, $J$ and $\nabla J$ are satisfied:
    \begin{enumerate}
        \item $P^2=\id$, \label{Pinvolutive}
        \item $g(PX,PY)=g(X,Y)$, \label{Pcompatible}
        \item $g(PX,Y)=g(X,PY)$, \label{Psymmetric}
        \item $PJ=-JP$, \label{PJanticommute}
        \item $P\nabla J(X,Y)+\nabla J(PX,PY)=0$, \label{PGprop}
    \end{enumerate}
    for all $X,Y$ smooth vector fields on $\Ss^3\times\Ss^3$.
\end{lemma}
Moreover, in \cite{propertiesofs3s3}, it was proved that the only three almost product structures that satisfy~$\eqref{curvtarures3s3}$ and all properties in Lemma \ref{proppe} are $P$, $-\frac{1}{2}P+\frac{\sqrt{3}}{2}JP$ and $-\frac{1}{2}P-\frac{\sqrt{3}}{2}JP$. 

The connected component of the identity of the isometry group of $\Ss^3\times\Ss^3$ is covered by $\Ss^3\times\Ss^3\times\Ss^3$,
where an element $\phi_{(a,b,c)}$ acts on a point $(p,q)$ by $\phi_{(a,b,c)}(p,q)=(apc^{-1},bqc^{-1})$.

% The sphere $\Ss^3$ is isomorphic to the Lie group $\SU(2)$, thus we will write $\isoo(\Ss^3\times\Ss^3)=\SU(2)\times\SU(2)\times\SU(2)$. 

This triple product is a double cover of the isometry group, since the two elements $(a,b,c)$ and $(-a,-b,-c)$ correspond to the same isometry.
Therefore, since $\Ss^3=\SU(2)$, 
\[\isoo(\Ss^3\times\Ss^3)=\frac{\SU(2)\times\SU(2)\times\SU(2)}{\Z_2},\]
which we will also denote as $\mathrm{P}(\SU(2)\times\SU(2)\times\SU(2))$.
%  Consequently, $\Ss^3\times\Ss^3$ is a homogeneous Riemannian manifold:
% \[
%     \Ss^3\times\Ss^3=\frac{\Ss^3\times\Ss^3\times\Ss^3}{\Delta\Ss^3},
% \]
% where $\Delta \Ss^3=\{(a,a,a):a\in\Ss^3\}$.

These isometries satisfy $d\phi_{(a,b,c)}\circ J=J\circ d\phi_{(a,b,c)}$ and $d\phi_{(a,b,c)}\circ P=P\circ d\phi_{(a,b,c)}$. 
% Denote by $\mathrm{SL}^{\pm}(2,\R)$ the group of all matrices in $M(2,\R)$ with determinant $\pm1$.
% We can write any matrix of $\mathrm{SL}^{\pm}(2,\R)$ as $\ii^k a$, where $\ii$ is the matrix given in~\eqref{ijksl2R}, $k\in\{0,1\}$ and $a\in\Ss^3$. 
% Thus, we have  
% \[
%     \left(\Ss^3\times\Ss^3\times\Ss^3\right)\rtimes\Z_2\subset\iso(\Ss^3\times\Ss^3).
% \]
By performing permutations of the elements of $\iso_o(\Ss^3\times\Ss^3)$, we obtain some isometries of the nearly Kähler $\Ss^3\times\Ss^3$, which are all in different connected components:
% \begin{tasks}[label=(\arabic*)](2)
% \task \ $\Psi_{0,0}(p,q)=(p,q)$, 
% \task \ $\Psi_{1,0}(p,q)=(q,p)$,
% \task \ $\Psi_{0,2\pi/3}(p,q)=(p q^{-1},q^{-1})$,
% \task \ $\Psi_{1,2\pi/3}(p,q)=(q^{-1},p q^{-1})$,
% \task \ $\Psi_{0,4\pi/3}(p,q)=(q p^{-1},p^{-1})$,
% \task \ $\Psi_{1,4\pi/3}(p,q)=(p^{-1},q p^{-1})$.
% \end{tasks}
\begin{equation}
    \label{isoslsl}
    \begin{alignedat}{2}
        &\Psi_{0,0}(p,q)=(p,q), 
        &&\Psi_{1,0}(p,q)=(q,p),\\
        &\Psi_{0,2\pi/3}(p,q)=(p q^{-1},q^{-1}),
        &&\Psi_{1,2\pi/3}(p,q)=(q^{-1},p q^{-1}),\\
        &\Psi_{0,4\pi/3}(p,q)=(q p^{-1},p^{-1}),\qquad\qquad
        &&\Psi_{1,4\pi/3}(p,q)=(p^{-1},q p^{-1}).
    \end{alignedat}
\end{equation}
Moreover, each one of these isometries satisfies
\[
    J \circ d\Psi_{\kappa,\tau}=(-1)^\kappa  d\Psi_{\kappa,\tau}\circ J,
% \]
% and
% \[
\ \ \ \ \ \ \ P\circ d \Psi_{\kappa,\tau}=d\Psi_{\kappa,\tau}\circ(\cos\tau P+\sin \tau J P).
\]
We aim to prove that these are all the isometries of $\Ss^3\times\Ss^3$.

Note that in \cite{butruille}, Proposition 3.1, Butruille shows the existence of a unique nearly Kähler structure on $\Ss^3\times\Ss^3$.
Therefore, the almost complex structure on $\Ss^3\times\Ss^3$ is unique up to sign.

The following lemma is a well known result.
\begin{lemma}\label{lemmasl2requaltoso21}
    Let $\{\alpha_1,\alpha_2,\alpha_3\}$ and $\{\beta_1,\beta_2,\beta_3\}$ be bases of $\mathrm{Im}(\mathbb{H})$. If $\li\alpha_i,\alpha_j\ri=\li \beta_i,\beta_j\ri$ for all $i$, $j\in \{1,2,3\}$, then there exists a element $c$ in $\Ss^3 \subset \mathbb{H}$ such that $c\alpha_i c^{-1}=\beta_i$. 
    In other words, $\Ss^3\cong\SU(2)$ is the double cover of $\SO(3)$.
\end{lemma}
With this lemma and Lemma \ref{proppe}, we prove the following statement.
% \begingroup
% \def\thetheorem{\ref{groupofisometries1}}

\begin{theorem}\label{groupofisometries}
    The isometry group of the nearly Kähler $\Ss^3\times\Ss^3$ is the semi-direct product $\mathrm{P}(\SU(2)\times\SU(2)\times\SU(2))\rtimes S_3$, where $S_3$ is the symmetric group of order 6 generated by $\{\Psi_{1,0},\Psi_{1,4\pi/3}\}$.
\end{theorem}
% \addtocounter{theorem}{-1}
% \endgroup

% \begin{theorem}\label{groupofisometries}
%     The isometry group of the nearly Kähler $\Ss^3\times\Ss^3$ is $\big(\Ss^3\times\Ss^3\times\Ss^3\big)\rtimes\big(\Z_2\times S_3\big)$, where $S_3$ is the symmetric group generated by $\{\Psi_{1,0},\Psi_{1,4\pi/3}\}$
% \end{theorem}
\begin{proof}
    We showed that $\mathrm{P}(\SU(2)\times\SU(2)\times\SU(2))\rtimes S_3\subset\iso(\Ss^3\times\Ss^3)$. 
    Now we show the other inclusion.

    Let $\mathcal{F}$ be an isometry of the nearly Kähler $\Ss^3\times\Ss^3$. 
    As the nearly Kähler structure of $\Ss^3\times\Ss^3$ is unique, there exists a $\kappa_0\in\{0,1\}$ satisfying
    \[
        \mathcal{F}_*J=(-1)^{\kappa_0}J\mathcal{F}_*.
    \]
    The tensor $\mathcal{F}_*P(\mathcal{F}^{-1})_*$ is an almost product structure satisfying all items in Lemma \ref{proppe} and Equation \eqref{curvtarures3s3},
    % , 
    % Lemma \ref{Junique} 
    so we have
    \[
        \mathcal{F}_*P(\mathcal{F}^{-1})_*=\cos \tau_0 P +\sin \tau_0 JP,
    \]
    for some $\tau_0\in\{0,\tfrac{2\pi}{3},\tfrac{4\pi}{3}\}$.
    By taking the composition $\mathcal{F}\circ \Psi_{\kappa_0,(-1)^{\kappa_0}\tau_0}$ we may assume that $\mathcal{F}$ preserves $P$ and $J$. 
    Let $(p_o,q_o)\in\Ss^3\times\Ss^3$ such that $\mathcal{F}(1,1)=(p_o,q_o)$. 
    Then by performing the composition $\mathcal{F}\circ\phi_{(p_o^{-1},q_{o}^{-1},1)}$ we assume $\mathcal{F}(1,1)=(1,1)$. 

    Suppose that ${\mathcal{F}_*}_{(1,1)}(\alpha,0)=(\beta,\gamma)$, for $\alpha,\beta,\gamma \in \mathrm{Im}(\quat)$. 
    As $\mathcal{F}$ preserves $P$, we get ${\mathcal{F}_*}_{(1,1)}(0,\alpha)=(\gamma,\beta)$. 
    Hence, we compute
    \begin{equation*}
        \begin{split}
            {\mathcal{F}_*}_{(1,1)}J(\alpha,0)&=\frac{1}{\sqrt{3}}{\mathcal{F}_*}_{(1,1)}(-\alpha,-2\alpha)\\
            &=-\frac{1}{\sqrt{3}}{\mathcal{F}_*}_{(1,1)}(\alpha,0)-\frac{2}{\sqrt{3}}{\mathcal{F}_*}_{(1,1)}(0,\alpha)\\
            &=-\frac{1}{\sqrt{3}}(\beta,\gamma)-\frac{2}{\sqrt{3}}(\gamma,\beta)\\
            &=-\frac{1}{\sqrt{3}}(\beta+2\gamma,2\beta+\gamma).
        \end{split}
    \end{equation*}
    On the other hand, since $\mathcal{F}$ preserves $J$, we have that this is equal to
    \begin{equation*}
        \begin{split}
            J{\mathcal{F}_*}_{(1,1)}(\alpha,0)&=J(\beta,\gamma)\\
            &=-\frac{1}{\sqrt{3}}(\beta-2\gamma,2\beta-\gamma).
        \end{split}
    \end{equation*}
    Consequently, we obtain $\gamma=0$. 
    Moreover, since $\mathcal{F}$ is an isometry, we know that ${\mathcal{F}_*}$ maps a set $\{(\alpha_1,0),(\alpha_2,0),(\alpha_3,0)\}$ to a set $\{(\beta_1,0),(\beta_2,0),(\beta_3,0)\}$ such that $\li \alpha_i,\alpha_j\ri=\li \beta_i,\beta_j\ri$.

    Now, using Lemma \ref{lemmasl2requaltoso21}, we may take the composition of $\mathcal{F}$ with an isometry of $\mathrm{P}\big(\SU(2)\times\SU(2)\times\SU(2)\big)$ such that we assume ${\mathcal{F}_*}_{(1,1)}(\alpha,0)=(\alpha,0)$ for all $\alpha\in\mathrm{Im}(\mathbb{H})$. 
    Since $\mathcal{F}$ preserves $P$, we have
    \begin{equation*}
        \begin{split}
            {\mathcal{F}_*}_{(1,1)}(\alpha,\beta)&={\mathcal{F}_*}_{(1,1)}(\alpha,0)+{\mathcal{F}_*}_{(1,1)}(0,\beta)\\
                                            &=(\alpha,0)+P(\beta,0)\\
                                            &=(\alpha,\beta).\\
        \end{split}
    \end{equation*}

   As isometries are determined by a point and its differential at that point, we showed that $\mathcal{F}^{-1}$ is in $\mathrm{P}\big(\SU(2)\times\SU(2)\times\SU(2)\big)\rtimes S_3$, thus $\mathcal{F}$ also belongs to this group.
\end{proof}

An element $(a,b,c,\Psi)\in\iso (\Ss^3\times\Ss^3)$  acts on a point $(p,q)$ by
    \begin{equation}\label{actionisometry}
        (a,b,c,\Psi)\cdot(p,q)=\Psi\circ\phi_{( a, b, c)}(p,q).
    \end{equation} 
The isometry group is a semidirect product with group law given by
\begin{equation*}
    \begin{split}
        (a_1,b_1,c_1,\Psi_1)&\star(a_2,b_2,c_2,\Psi_2)      =\left(\sigma_2(a_1,b_1,c_1)\cdot(a_2,b_2,c_2),\Psi_1\circ\Psi_2\right) 
    \end{split}
\end{equation*}
where $\sigma_2$ is the permutation of $(1,2,3)$ in $S_3$ associated to $\Psi_2$.

\section{The nearly Kähler \texorpdfstring{$\CP$}{ℂP³}}

In \cite{liefsoens2024thesis}, it is demonstrated how to describe the nearly Kähler $\CP$ starting from the Hopf fibration $\pi: \Ss^7 \to \CP$, as will also appear soon in \cite{Liefsoens2024}.
We recall the basic facts of this description. Starting from the round $\Ss^7 \subset \C^4$, the Hopf fibration is a Riemannian submersion by putting the Fubini-Study metric $g_\circ$ on $\CP$, and $\C^4$ induces a Kähler almost complex structure $\Jo$ on $\CP$. 
By embedding $\Ss^7 \subset \quat^2$ instead, we can consider the vertical distribution $V(p) = \vct\{i p\}$, and the distributions $\tilde{\mathcal{D}}^2(p) = \vct\{ j p, k p \}$ and $\tilde{\mathcal{D}}^4 = (V \oplus \tilde{\mathcal{D}}^2)^\perp$ on $\Ss^7$. 
In \cite{liefsoens2024thesis}, it is shown that $\mathcal{D}^2 = \pi \circ \tilde{\mathcal{D}}^2$ is well defined, and that an analogous statement holds for $\mathcal{D}^4$. 
By defining an almost product structure $P$ on $\CP$ that acts as minus the identity on $\mathcal{D}^2$ and as the identity on $\mathcal{D}^4$, we can define a new almost product structure $J = P \Jo = \Jo P$. 
By rescaling the Fubini-Study metric on $\mathcal{D}^4$ with a factor $2$, we then get a new metric $g$ such that $(\CP, g, J)$ is a nearly Kähler space. 

Furthermore, in \cite{liefsoens2024thesis}, an expression is given for the Riemann curvature of the nearly Kähler $\CP$. 
By looking at the sectional curvature of the plane spanned by $\{X, \Jo X\}$, it is shown that the almost product structure $P$ is preserved by every isometry of nearly Kähler $\CP$.
Moreover, just as noted above, the classification of Butruille also gives that $J$ is preserved (up to sign) by each isometry. Since $\Jo = JP$, we then find that also $\Jo$ is preserved (up to sign) by each isometry of the nearly Kähler $\CP$. 
In other words, these structures are all geometric invariants of the nearly Kähler $\CP$.

Since the almost product structure is preserved, we find in particular that the two distributions $\mathcal{D}^2(p)$ and $\mathcal{D}^4(p)$ are preserved by each isometry. 
By linking the Kähler and nearly Kähler metric, it was then found in \cite{Liefsoens2024} that the isometries of the nearly Kähler $\CP$ are exactly those isometries of Kähler $\CP$ that preserve the almost product structure $P$. 
The full isometry group of Kähler $(\CP, g_\circ, \Jo)$ is well-known and given by
\[ \PU(4) \rtimes \Z_2 = \{ \varphi \mid \varphi \circ \pi = \pi \circ A \text{ and } A \in \U(4) \} \rtimes \Z_2, \]
where the action of $\Z_2$ comes from the induced action of complex conjugation on $\C^4$. 
Since complex conjugation is an isometry of the nearly Kähler $\CP$ and it has to lie within $\PU(4) \rtimes \Z_2$, we just need the identity component of the isometry group to know the full isometry group. 
It is well-known that the identity component of the isometry group is $\PSp(2)$ (see for instance \cite{onishchik}). 
This can be shown by explicitly considering the elements of $\PU(4)$ that preserve the almost product structure, and can also be seen with an algebraic reasoning, such as the one presented in \cite{onishchik}. This reasoning then gives the following result. 
\begin{theorem}
    The isometry group of the nearly Kähler $\CP$ is \[ \PSp(2) \rtimes \Z_2 =  (\Sp(2) / \Z_2) \rtimes \Z_2. \] 
\end{theorem}

Explicitly, an element $(A,k)$ acts on a point $\pi(p) \in \CP$ by
\begin{align}\label{actionisometry_CP3}
    (A,k)\cdot \pi(p) &= \pi( \mathrm{Conj}^k(A p) ).
\end{align}
Here, $\mathrm{Conj}^k$ is complex conjugation if $k=1$ and is just the identity if $k=0$.
The isometry group is a semidirect product since for the action in \eqref{actionisometry_CP3} to be a Lie group action, the group law has to be
\begin{equation*}
    \begin{split}
        (A,k_1) &\star (B , k_2) = ( \mathrm{Conj}^{k_2}(A) \, B, k_1 + k_2 ) .
    \end{split}
\end{equation*} 

\section{The nearly Kähler flag manifold}

Let $V$ be an $n$-dimensional $\mathbb{F}$-vector space with $\mathbb{F}=\R,\C$. 
Define the manifold of flags in $V$ as the set
\[
F_{i_1,\ldots,i_k}(V)=\{(V_{i_1},\ldots,V_{i_k}):V_{i_1}\subset\ldots\subset V_{i_k}\subset V\}
\]
where $0<i_1<\ldots<i_k<n$, and $V_{i_j}$ are $\mathbb{F}$-vector subspaces of $V$ of dimension $i_j$ for all $j=1,\ldots,k$.
If $k=n$ we denote it by $F(V)$ and we call it the manifold of \textit{full} flags in $V$. 
The flag manifold is a homogeneous manifold. 
Namely, given an inner product $\li,\ri$ on $V$ we take $\mathrm{O}(V)$ as the linear transformations of $V$ that preserve $\li,\ri$.
Then $\mathrm{O}(V)$ is a Lie group that acts transitively on $F_{i_1,\ldots,i_k}(V)$.

Now we focus on $F(\C^3)=F_{1,2}(\mathbb{C}^3)$. We denote the elements of $F(\C^3)$ as $(\ell,\Pi)$, with $\ell$ and $\Pi$ a complex line and a complex plane, respectively, satisfying $\ell\subset\Pi$.
Let $\li\li,\ri\ri$ be the Hermitian inner product in $\C^3$, that is 
\[
    \li\li x,y\ri\ri=\sum_{i=1}^3x_i\overline{y_i}.
\]
We denote by $\ell^\perp$ and $\Pi^\perp$ the orthogonal complex plane and orthogonal complex line with respect to $\li\li,\ri\ri$, respectively.

There exist three canonical submersions onto $\C P^2$ (see Figure \ref{lines}): 
\[
\pi_1\colon(\ell,\Pi)\mapsto \Pi^\perp, \hspace{2 cm} \pi_2\colon(\ell,\Pi)\mapsto\ell^\perp \cap \Pi , \hspace{2 cm} \pi_3\colon(\ell,\Pi)\mapsto \ell.
\]
\begin{figure}[ht]
    % \centering
    \small
    \def\svgwidth{0.4\columnwidth}
    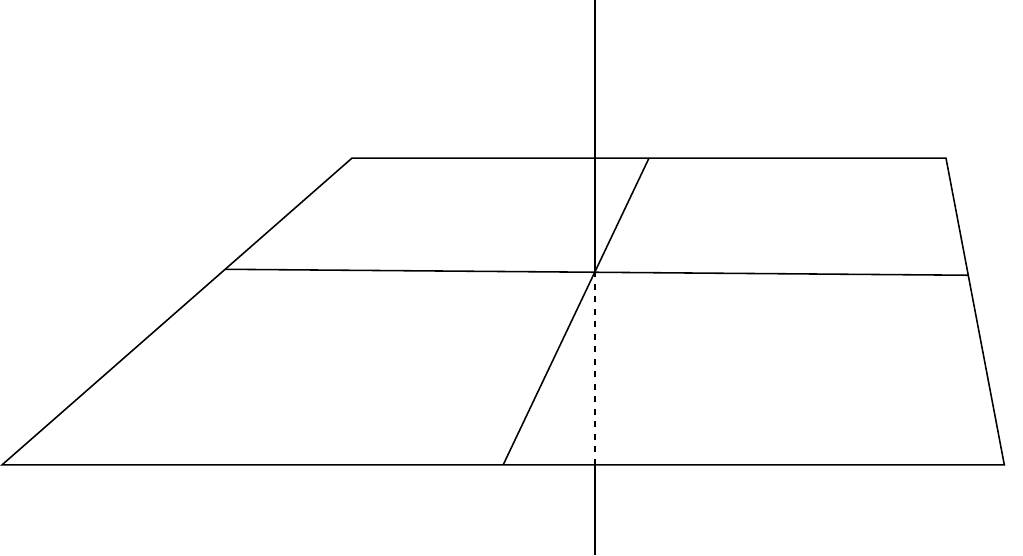

    \caption{The three possible types of lines in $\mathbb{C}P^2$ we can obtain from a flag $(\ell, \pi)$.}
    \label{lines}
\end{figure} 

The fibers of these three immersions are diffeomorphic to $\Ss^2$. We see it only for $\pi_3$.
The fiber over a line $\ell_o\in\C P^2$ is the subset $\pi^{-1}(\{\ell_o\})=\{(\ell_o,\Pi):\ell_o\subset\Pi\}$ of $F(\C^3)$. 
Take $\Pi$ containing $\ell_o$, then $\ell_o$ and the line $\ell_o^\perp\cap\Pi$ span $\Pi$. 
Hence, we can think of the fiber over $\ell_o$ as the set of lines in $\ell^\perp_o$, which is diffeomorphic to $\C P^1\cong\Ss^2$ (see Figure \ref{fiber}). 

\begin{figure}[ht]
    \scriptsize
    \def\svgwidth{0.4\columnwidth}
    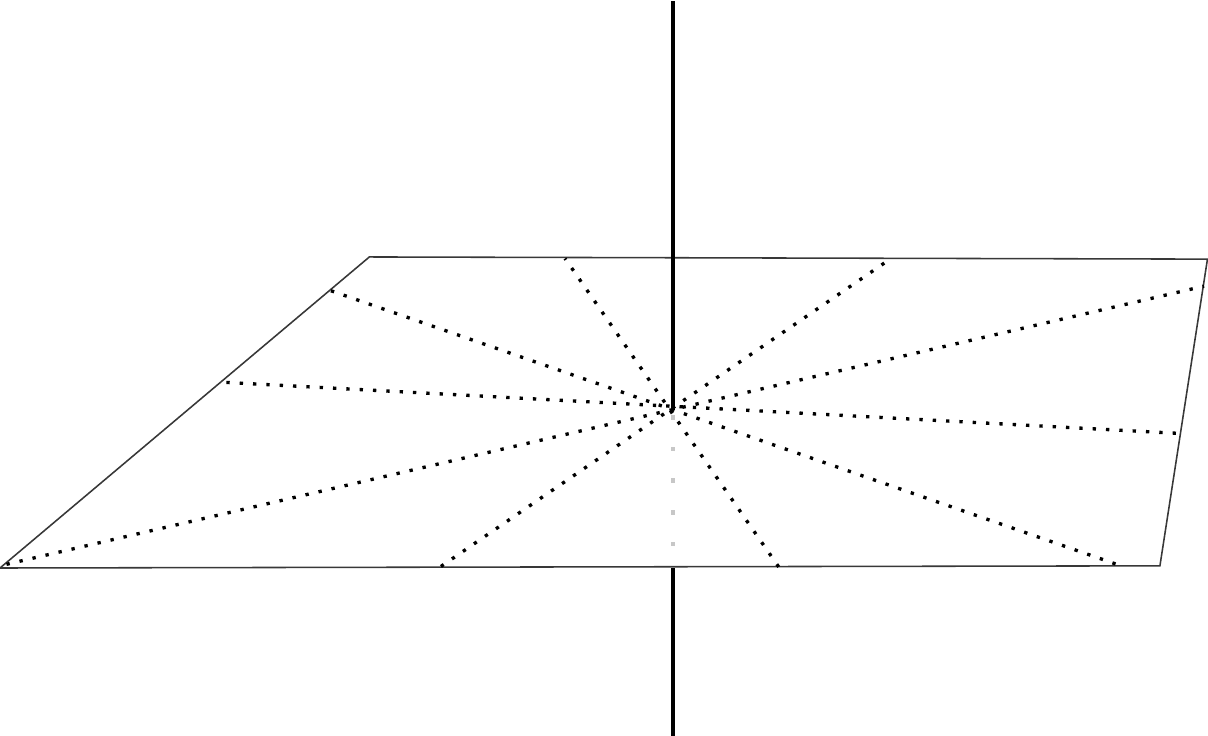

   \caption{A representation of all the elements of the fiber, projected onto the plane $\ell^\perp_o$.}
   \label{fiber}
\end{figure}

Take the Fubini-Study metric on $\CP[2]$. Then there exists a unique metric $g$ (up to homotheties) on $F(\C^3)$ such that $\pi_1$, $\pi_2$ and $\pi_3$ are all Riemannian submersions.

Let $V_i$ the vertical distribution on $F(\C^3)$ associated to $\pi_i$, that is $V_i=d\pi_i^{-1}(\{0\})$. 
We can easily see that $TF(\C^3)=V_1\oplus V_2 \oplus V_3$. Identifying $V_i\cong T\Ss^2$, we assign the round metric $\li,\ri_i$ at each factor. 
The metric $g$ on $F(\C^3)$ is given by 
\[
        g(X,Y)=\li X_1,Y_1\ri_1+\li X_2,Y_2 \ri_2+\li X_3,Y_3\ri_3
\]  
We assign the same orientation to the fibers of $\pi_1$, $\pi_2$, and $\pi_3$. 
Then we define an almost complex structure $J$ on $F(\C^3)$ as a clockwise rotation in $\tfrac{\pi}{2}$ on each distribution $V_i$. 
Finally, we have that $(F(\C^3),g,J)$ is a nearly Kähler manifold. 

Moreover, the manifolds $(F(\C^3),g_i,J_i)$, with
\begin{equation}
g_i=\begin{cases}
    2g & \text{on $V_i$},\\
    g & \text{the rest},\\
\end{cases}
\ \ \ \ \ \ \ \ \ \ \ \ \ \ \ \ \ 
J_i=\begin{cases}
    J & \text{on $V_i$},\\
    -J & \text{the rest},\\
\end{cases}\label{kahlerstructures}
\end{equation}
are Kähler manifolds. 
The inverse transformation from Kähler to nearly Kähler was noted by Hitchin \cite{hitchin}.

Consider the following six natural diffeomorphisms from $F(\C^3)$ to itself:
\begin{equation}
    \begin{aligned}
        \phi_0\colon (\ell,\Pi)&\mapsto (\ell,\Pi), \ \ \ \ &&\phi_1\colon(\ell,\Pi)\mapsto(\ell^\perp\cap\Pi,\Pi), \ \ \ \ &&\phi_2\colon(\ell,\Pi)\mapsto(\Pi^\perp,\ell^\perp) , \\
        \phi_3\colon(\ell,\Pi)&\mapsto(\ell,(\ell^\perp\cap\Pi)^\perp), \ \ \ \ &&\phi_4\colon(\ell,\Pi)\mapsto(\ell^\perp\cap\Pi,\ell^\perp), \ \ \ \ &&\phi_5\colon(\ell,\Pi)\mapsto (\Pi^\perp,(\ell^\perp\cap\Pi)^\perp).
    \end{aligned}   \label{isometries} 
\end{equation}
These six maps are isometries with respect to the nearly Kähler metric and together form a finite group, isomorphic to $S_3$. 
On the other hand, besides the identity, $\phi_i$ is the only isometry with respect to the Kähler metric $g_i$.

The complex conjugation on $\C^3$ yields the map $\Psi:(\ell,\Pi)\mapsto(\overline{\ell},\overline{\Pi})$, which is an isometry with respect to $g$, $g_1$, $g_2$ and $g_3$.
\begin{theorem}\label{theoremisometrygroup}
    The isometry group of the nearly Kähler $F(\C^3)$ is $\PSU(3)\times S_3\rtimes \Z_2$.
\end{theorem}
In \cite{shankar2001}, it is shown that the isometry group of the Kähler $F(\C^3)$ is $\PSU(3)\rtimes (\Z_2\oplus\Z_2)$.

To prove this theorem, we first study $F(\C^3)$ as a homogeneous space.
% As mentioned before, the Lie group of linear transformations preserving $\li\li,\ri\ri$ acts transitively on $F(\C^3)$. 
% In our case, that group is $\U(3)$. Moreover, this group acts by isometries, with isotropy subgroup $\U(1)\times\U(1)\times\U(1)$. However, some of these isometries do not preserve the almost complex structure $J$, so w
We take $\SU(3)$, which acts transitively on $F(\C^3)$ by isometries, and preserving $J$. Moreover, this group preserves $J_i$ and $V_i$, for $i=1,2,3$.

We have 
\[
F(\C^3)=\frac{\SU(3)}{\U(1)\times\U(1)}.
\]
The reductive decomposition of $\su(3)$ is given by $\mathfrak{h}\oplus\mathfrak{m}$, where $\mathfrak{h}$ is the Lie algebra of $\U(1)\times\U(1)$, and $\mathfrak{m}$ is the complement of $\mathfrak{h}$ spanned by 
\begin{equation}
    \begin{split}
m_1=\begin{pmatrix}
    0 & -1 & 0\\
    1 & 0 & 0\\
    0 & 0 & 0\\
\end{pmatrix}, \ \ \ \ \ \ 
m_2=\begin{pmatrix}
    0 & i & 0\\
    i & 0 & 0\\
    0 & 0 & 0\\
\end{pmatrix}, \\
m_3=\begin{pmatrix}
    0 & 0 & 1\\
    0 & 0 & 0\\
    -1 & 0 & 0\\
\end{pmatrix}, \ \ \ \ \ \ 
m_4=\begin{pmatrix}
    0 & 0 & i\\
    0 & 0 & 0\\
    i & 0 & 0\\
\end{pmatrix}, \\
m_5=\begin{pmatrix}
    0 & 0 & 0\\
    0 & 0 & -1\\
    0 & 1 & 0\\
\end{pmatrix}, \ \ \ \ \ \ 
m_6=\begin{pmatrix}
    0 & 0 & 0\\
    0 & 0 & i\\
    0 & i & 0\\
\end{pmatrix}. \\
\end{split}
\end{equation}
We denote by $\mathfrak{m}_1=\mathrm{Span}\{m_1,m_2\}$, $\mathfrak{m}_2=\mathrm{Span}\{m_3,m_4\}$ and $\mathfrak{m}_3=\mathrm{Span}\{m_5,m_6\}$.

The map $\pi\colon\SU(3)\to F(\C^3):u\mapsto u \U(1)\times\U(1)$ is a Riemannian submersion, where the metric on $\SU(3)$ is a multiple of the Killing form:
\[
        \li x,y\ri=-\frac{1}{2}\Tr(xy).
\]  
The left translation of $\mathfrak{m}_i$ yields the distribution $V_i$, and the translations of the vectors $m_i$ form an orthonormal frame on $F(\C^3)$.

Let $(\ell,\Pi)\in F(\C^3)$ with $\ell=\{t v_1:t\in\R\}$ for some $v_1 \in\C^3$, and $\Pi=\{t v_1+sv_2:t,s\in\R\}$, where $v_2$ is an element of $\C^3$ such that $\li\li v_1,v_2\ri\ri=0$. It corresponds to the class \[
(\ell,\Pi)=[(v_1\vert v_2\vert v_3)],
\]
in $\SU(3)/\U(1)\times\U(1)$, 
where $v_3$ is orthogonal to $v_1$ and $v_2$ with respect to $\li\li,\ri\ri$.

% The flag manifold $F(\C^3)$ is a naturally reductive space. 
% That is, $\mathrm{Ad}(h)\mathfrak{m}\subset\mathfrak{m}$ for all $h\in \U(1)\times\U(1)$ and $\li [X,Y]_\mathfrak{m},Z\ri=\li X,[Y,Z]_\mathfrak{m}\ri$ for all $X,Y,Z\in \mathfrak{m}$. Naturally reductive spaces are geodesic homogeneous spaces, that is, geodesics are given by orbits of monoparametric subgroups. Namely, $\gamma$ is a geodesic of $F(\C^3)$ through $p$ if and only if 
% \[
%         \gamma(t)=\exp(t X)\cdot p,
% \]  
% for $X\in\mathfrak{m}$.

By identifying $\mathfrak{m}$ with $T_{[\id_3]} F(\C^3)$, we can give the almost complex structure $J$ by defining
\begin{equation*}
    Jm_1=m_2, \ \ \ \ \ \ \ Jm_3=m_4, \ \ \ \ \ \ \ Jm_5=m_6.
\end{equation*}
The other complex structures $J_1$, $J_2$ and $J_3$ can be written similarly using \eqref{kahlerstructures}.

% Another advantage of naturally reductive spaces is that they have a neat expression for their curvature tensor.
% That is, the curvature tensor at the class of the identity element of a naturally reductive space $G/H$, with reductive decomposition $\mathfrak{g}=\mathfrak{h}\oplus\mathfrak{m}$ is given by

% \[
%     R_o(X,Y)Z=\frac{1}{2}[Z, [X, Y ]_\mathfrak{m}]_\mathfrak{m} - [[X, Y ]_\mathfrak{h}
%     , Z]_\mathfrak{m}  +\frac{1}{4}[[Z, Y ]_\mathfrak{m}, X]_\mathfrak{m}-\frac{1}{4}[[Z, X]_\mathfrak{m}, Y ]_\mathfrak{m}
% \]
In \cite{kamil}, it is shown that the curvature tensor is given by
\begin{equation}\label{curvaturetensor}
    \begin{split}
        R(X,Y)Z&=\frac{1}{4}(g(Y,Z)X - g(X,Z)Y )\\
        &\quad- \frac{1}{4} (g(J  Y,Z)J  X - g(J  X,Z)J  Y + 2g(X, J  Y )J  Z)\\
        &\quad+ \frac{1}{2} (g(J_1Y,Z)J_1X - g(J_1X,Z)J_1Y + 2g(X, J_1Y )J_1Z)\\
        &\quad+ \frac{1}{2} (g(J_2Y,Z)J_2X - g(J_2X,Z)J_2Y + 2g(X, J_2Y )J_2Z)\\
        &\quad+ \frac{1}{2} (g(J_3Y,Z)J_3X - g(J_3X,Z)J_3Y + 2g(X, J_3Y )J_3Z).
    \end{split}
\end{equation}

Let $(i,j,k)\in S_3$ be a permutation of three elements. 
We identify the isometries given in \eqref{isometries} as
\begin{equation}
\begin{split}
    \phi_0=(1,2,3),\ \  \phi_1=(2,1,3),\ \  \phi_2=(3,2,1),\\
     \phi_3=(1,3,2), \ \ \phi_4=(2,3,1), \ \ \phi_5=(3,1,2).
\end{split}  \label{identifications3}  
\end{equation}
Then, these isometries map elements in $\SU(3)/\U(1)\times\U(1)$ to permutations of the columns of the matrices in $\SU(3)$. That is,
\[
\phi_i[(v_1\vert v_2 \vert v_3)]=[(v_{\sigma^i_1}\vert \sign(\sigma^i) v_{\sigma^i_2}\vert v_{\sigma^i_3})],
\]
where $\sigma^i$ is the permutation associated to $\phi_i$. The sign of the permutation is added to preserve the orientation, and consequently, the determinant.

On the other hand, the conjugation map $\Psi$ acts on elements of $\SU(3)/\U(1)\times\U(1)$ by
\[
    \Psi[(v_1\vert v_2\vert v_3)]=[(\overline{v_1}\vert\overline{v_2}\vert\overline{v_3})].
\] 

It is known (see for instance \cite{butruille}) that the flag manifold carries a unique (up to sign and homotheties) nearly Kähler structure. 
That is, if $\tilde{g}$ and $\tilde{J}$ are a Riemannian metric and an almost complex structure that satisfies $\tilde{\nabla}\tilde{J}(X,Y)+\tilde{\nabla}\tilde{J}(Y,X)=0$ for all $X,Y\in\mathfrak{X}(F(\C^3))$ then $\tilde{g}=\lambda g$ and $\tilde{J}=\pm J$, for $\lambda>0$.

Note that given any isometry $\varphi$, then $\dd \varphi J \dd \varphi^{-1}$ is an almost complex structure. Therefore, as isometries preserve the Levi-Civita connection, we always have $d\varphi J=\pm Jd\varphi$.
By computing explicitly, we see
\begin{equation*}
    \begin{aligned}
d\phi_1J&=-Jd\phi_1, \hspace*{10 mm}&&d\phi_1J_1=-J_1d\phi_1, \hspace*{10 mm}&&d\phi_1 J_2=-J_3d\phi_1, \hspace*{10 mm}&&d\phi_1J_3=-J_2d\phi_1,\\
d\phi_2J&=-Jd\phi_2, \hspace*{10 mm}&&d\phi_2J_1=-J_3d\phi_2, \hspace*{10 mm}&&d\phi_2 J_2=-J_2d\phi_2, \hspace*{10 mm}&&d\phi_2J_3=-J_1d\phi_2,\\
d\phi_3J&=-Jd\phi_3, \hspace*{10 mm}&&d\phi_3J_1=-J_2d\phi_3, \hspace*{10 mm}&&d\phi_3 J_2=-J_1d\phi_3, \hspace*{10 mm}&&d\phi_3J_3=-J_3d\phi_3,\\
d\phi_4J&= Jd\phi_4, \hspace*{10 mm}&&d\phi_4J_1= J_2d\phi_4, \hspace*{10 mm}&&d\phi_4 J_2= J_3d\phi_4, \hspace*{10 mm}&&d\phi_4J_3= J_1d\phi_4,\\
d\phi_5J&= Jd\phi_5, \hspace*{10 mm}&&d\phi_5J_1= J_3d\phi_5, \hspace*{10 mm}&&d\phi_5 J_2= J_1d\phi_5, \hspace*{10 mm}&&d\phi_5J_3= J_2d\phi_5,\\
d\Psi  J&=-Jd\Psi  , \hspace*{10 mm}&&d\Psi  J_1=-J_1d\Psi  , \hspace*{10 mm}&&d\Psi   J_2=-J_2d\Psi  , \hspace*{10 mm}&&d\Psi  J_3=-J_3d\Psi  .\\
    \end{aligned}
\end{equation*}
We also have that for $i=1,2,3$, the map $\phi_i$ preserves $V_i$ and swaps the remaining two distributions, while $\phi_4$ and $\phi_5$ permute the distributions according to their permutations in $S_3$ given in \eqref{identifications3}. 
Moreover, $\Psi$ preserves the three distributions.

\begin{proof}[Proof of Theorem \ref{theoremisometrygroup}]
    From the group structure induced from the composition of  $\phi_i$ for $i=1,\ldots,5$ and $\Psi$ we already know that $\SU(3)\times S_3\rtimes \Z_2$ is contained in $\iso(F(\C^3))$. 
    Now we prove the other inclusion.

    Let $\varphi$ be an isometry of the nearly Kähler $F(\C^3)$. 
    From \eqref{curvaturetensor} we see that the curvature of a holomorphic plane $\mathrm{Span}\{X,JX\}$ is given by
    \[
    g(R(X,JX)JX,X)=-\frac{1}{2}+\frac{3}{2}\left(g(JJ_1X,X)^2+g(JJ_2X,X)^2+g(JJ_3X,X)^2\right).
    \]
By Cauchy-Schwartz inequality this is less than or equal to 4 and the equality holds if and only if $X$ belongs to $V_i$ for some $i\in\{1,2,3\}$.

Take $X\in V_i$, for $i\in\{1,2,3\}$. 
Then, since $\varphi$ must preserve (up to sign) the almost complex structure $J$, it follows
\begin{equation*}
    \begin{split}
        4&=g(R(X,JX)JX,X)=g(\varphi_*R(X,JX)JX,\varphi_*X)\\
        &=g(R(\varphi_*X,J\varphi_*X)J\varphi_*X,\varphi_*X)\\
        &=-\frac{1}{2}+\frac{3}{2}(g(JJ_1\varphi_*X,\varphi_*X)^2+g(JJ_2\varphi_*X,\varphi_*X)^2+g(JJ_3\varphi_*X,\varphi_*X)^2)
    \end{split}    
\end{equation*}
Then $\varphi_*X$ must belong to $V_j$ for some $j\in\{1,2,3\}$. 
This means $\varphi_*V_i=V_j$. 
% 
% Hence, as $\varphi_*$ preserves $G$, 
Therefore, $\varphi_*$ preserves the partition $V_1\oplus V_2\oplus V_3$.

Now, by doing a composition with one of the isometries in \eqref{isometries}, we may assume that $\varphi_*V_i=V_i$ for all $i=1,2,3$.
Moreover, by doing a translation to the class of the identity  element, we may assume that $\varphi([\id_3])=[\id_3]$.

Now take a basis $\{X_1,JX_1,X_2,JX_2,X_3,JX_3\}$ of $T_{[\id_3]}F(\C^3)$ with $X_i\in V_i$ for all $i=1,2,3$. 
As $\varphi_*$ preserves the distributions $V_i$, we have
\begin{equation*}
    % \begin{aligned}
        \varphi_*X_i=\cos\theta_i X_i+\sin\theta_iJX_i, \ \ \ \ \ \ \ \ \varphi_*JX_i=\pm\cos\theta_i JX_i\mp\sin\theta_iX_i.
    % \end{aligned}
\end{equation*}
Note that from the properties of a nearly Kähler manifold the tensor $ \nabla J$ behaves like a cross product (see \cite{bolton}). It follows $\nabla J(X_i,X_j)\in \varepsilon_{ijk}V_k$ for all $i,j,k=1,2,3$.
Then we may take $X_3=\nabla J(X_1,X_2)$.
Since isometries preserve $\nabla J$, we deduce that $\theta_3=-(\theta_1+\theta_2)$.

Taking a composition with the isometry in the isotropy subgroup $\U(1)\times\U(1)$ given by
\[
\begin{pmatrix}
    e^{\frac{1}{3} i (\theta_1-\theta_2 )} & 0 & 0 \\
    0 & e^{-\frac{1}{3} i (2 \theta_1+\theta_2  )} & 0 \\
    0 & 0 & e^{\frac{1}{3} i (\theta_1 +2 \theta_2 )} \\
\end{pmatrix}
\]
we may assume that $\theta_1=\theta_2=0$.
Finally we have
\begin{equation*}
    % \begin{aligned}
        \varphi_*X_i=X_i, \ \ \ \ \ \ \ \ \varphi_*JX_i=\pm JX_i,
    % \end{aligned}
\end{equation*}
for all $i=1,2,3$.
Then, $\varphi$ must be equal to either the identity map or to the conjugation map~$\Psi$. 

Given that we always took compositions with elements in $\SU(3)\times S_3\rtimes\Z_2$ and the final result is an element of this group, the original isometry $\varphi$ is included in this group.
\end{proof}
We can also think of the action of a permutation $\sigma$ on a element $[v]$ as $[vP^\sigma]$, where $P^\sigma=\sigma(\id_3)$ is the associated permutation matrix. 
The action of an isometry $(A,\sigma,k)$ on an element $[v]$ is $[\mathrm{Conj}^k(AvP^\sigma)]$ where $\mathrm{Conj}$ denotes complex conjugation, as before. 
Therefore, the product rule on the isometry group is a semidirect product, given by
\[
(A_1,\sigma_1,k_1)\star(A_2,\sigma_2,k_2)=(\mathrm{Conj}^{k_2}(A_1)A_2,\sigma_1\sigma_2,k_1+k_2).
\]

\bibliographystyle{abbrv}
\bibliography{main}
\end{document}

%% file: data.tex
\title{Isometry groups of nearly Kähler manifolds}
\date{}
\author{Mateo Anarella and Michaël Liefsoens}
\address{M. Anarella, M. Liefsoens, Address: Department of mathematics, KU Leuven, Celestijnenlaan 200B 3001, Leuven, Belgium.}

\email{mateo.anarella@kuleuven.be}
\email{michael.liefsoens@kuleuven.be}
\thanks{M. Anarella has been partially supported by FWO and FNRS under EOS project G0I2222N. M. Liefsoens is supported by the Research Foundation Flanders (FWO) with project 11PG324N. }
\keywords{Nearly Kähler, homogeneous manifold, six-dimensional, isometry group}
\subjclass[2020]{53C15, 53C30}

%% file: packages.tex
\usepackage{amsmath,amssymb,amsfonts,mathtools,amsthm}
\usepackage{amsfonts}
\usepackage{graphicx}
\usepackage{graphics}
\usepackage{amstext} 
\usepackage[italicdiff]{physics}
\usepackage{mathrsfs}

\usepackage[hidelinks,pdfencoding=auto,psdextra]{hyperref}
\usepackage{color,soul}
\usepackage[dvipsnames]{xcolor}

\usepackage{float}
\usepackage{tabularray}
\usepackage{tasks}
\usepackage{tikz-cd}
\usepackage{tikz}
\usetikzlibrary{graphs,quotes}

\usepackage{geometry}
\setul{}{0.3 ex}

\newtheorem{theorem}{Theorem}

\newtheorem{lemma}[theorem]{Lemma}

\theoremstyle{definition}

%% file: commands.tex
\newcommand{\R}{\mathbb{R}}
\newcommand{\Z}{\mathbb{Z}}
\newcommand{\C}{\mathbb{C}}
\newcommand{\quat}{\mathbb{H}}
\newcommand{\Sl}{\mathrm{SL}(2,\R)}
\newcommand{\id}{\operatorname{Id}}
\newcommand{\li}{\langle}
\newcommand{\ri}{\rangle}

\newcommand{\CP}[1][3]{\C P^#1}
\DeclareMathOperator{\vct}{Span}
\newcommand{\Jo}{J_\circ}

\newcommand{\sign}{\operatorname{sign}}

\newcommand{\Ss}{\mathbb{S}}

\newcommand{\SO}{\mathrm{SO}}

\newcommand{\U}{\mathrm U}
\newcommand{\SU}{\mathrm{SU}}

\newcommand{\su}{\mathfrak{su}}
\newcommand{\Sp}{\mathrm{Sp}}
\newcommand{\PSp}{\mathrm{PSp}}
\newcommand{\PU}{\mathrm{PU}}
\newcommand{\PSU}{\mathrm{PSU}}
\newcommand{\Ort}{\mathrm{O}}

\newcommand{\iso}{\operatorname{Iso}}
\newcommand{\isoo}{\operatorname{Iso}_o}

\makeatletter
\newcommand\columntag[2]{#1\def\@currentlabel{#1}.\label{#2}}
\makeatother

\usepackage{import}
\usepackage{xifthen}
\usepackage{pdfpages}
\usepackage{transparent}

%% file: images/lines.pdf_tex
%% Creator: Inkscape 1.3.2 (091e20e, 2023-11-25, custom), www.inkscape.org
%% PDF/EPS/PS + LaTeX output extension by Johan Engelen, 2010
%% Accompanies image file 'lines.pdf' (pdf, eps, ps)
%%
%% To include the image in your LaTeX document, write
%%   \input{<filename>.pdf_tex}
%%  instead of
%%   \includegraphics{<filename>.pdf}
%% To scale the image, write
%%   \def\svgwidth{<desired width>}
%%   \input{<filename>.pdf_tex}
%%  instead of
%%   \includegraphics[width=<desired width>]{<filename>.pdf}
%%
%% Images with a different path to the parent latex file can
%% be accessed with the `import' package (which may need to be
%% installed) using
%%   \usepackage{import}
%% in the preamble, and then including the image with
%%   \import{<path to file>}{<filename>.pdf_tex}
%% Alternatively, one can specify
%%   \graphicspath{{<path to file>/}}
%% 
%% For more information, please see info/svg-inkscape on CTAN:
%%   http://tug.ctan.org/tex-archive/info/svg-inkscape
%%
\begingroup%
  \makeatletter%
  \providecommand\color[2][]{%
    \errmessage{(Inkscape) Color is used for the text in Inkscape, but the package 'color.sty' is not loaded}%
    \renewcommand\color[2][]{}%
  }%
  \providecommand\transparent[1]{%
    \errmessage{(Inkscape) Transparency is used (non-zero) for the text in Inkscape, but the package 'transparent.sty' is not loaded}%
    \renewcommand\transparent[1]{}%
  }%
  \providecommand\rotatebox[2]{#2}%
  \newcommand*\fsize{\dimexpr\f@size pt\relax}%
  \newcommand*\lineheight[1]{\fontsize{\fsize}{#1\fsize}\selectfont}%
  \ifx\svgwidth\undefined%
    \setlength{\unitlength}{490.94556914bp}%
    \ifx\svgscale\undefined%
      \relax%
    \else%
      \setlength{\unitlength}{\unitlength * \real{\svgscale}}%
    \fi%
  \else%
    \setlength{\unitlength}{\svgwidth}%
  \fi%
  \global\let\svgwidth\undefined%
  \global\let\svgscale\undefined%
  \makeatother%
  \begin{picture}(1,0.54302489)%
    \lineheight{1}%
    \setlength\tabcolsep{0pt}%
    \put(0,0){\includegraphics[width=\unitlength,page=1]{lines.pdf}}%
    \put(0.90352753,0.40066053){\makebox(0,0)[t]{\lineheight{1.04999995}\smash{\begin{tabular}[t]{c}$\Pi$\end{tabular}}}}%
    \put(0.45920933,0.11982932){\makebox(0,0)[t]{\lineheight{1.04999995}\smash{\begin{tabular}[t]{c}$\ell$\end{tabular}}}}%
    \put(0.64616218,0.50867858){\makebox(0,0)[t]{\lineheight{1.04999995}\smash{\begin{tabular}[t]{c}$\Pi^\perp$\end{tabular}}}}%
    \put(0.84278045,0.21986802){\makebox(0,0)[t]{\lineheight{1.04999995}\smash{\begin{tabular}[t]{c}$\ell^\perp\cap\Pi$\end{tabular}}}}%
  \end{picture}%
\endgroup%

%% file: images/fiber.pdf_tex
%% Creator: Inkscape 1.3.2 (091e20e, 2023-11-25, custom), www.inkscape.org
%% PDF/EPS/PS + LaTeX output extension by Johan Engelen, 2010
%% Accompanies image file 'fiber.pdf' (pdf, eps, ps)
%%
%% To include the image in your LaTeX document, write
%%   \input{<filename>.pdf_tex}
%%  instead of
%%   \includegraphics{<filename>.pdf}
%% To scale the image, write
%%   \def\svgwidth{<desired width>}
%%   \input{<filename>.pdf_tex}
%%  instead of
%%   \includegraphics[width=<desired width>]{<filename>.pdf}
%%
%% Images with a different path to the parent latex file can
%% be accessed with the `import' package (which may need to be
%% installed) using
%%   \usepackage{import}
%% in the preamble, and then including the image with
%%   \import{<path to file>}{<filename>.pdf_tex}
%% Alternatively, one can specify
%%   \graphicspath{{<path to file>/}}
%% 
%% For more information, please see info/svg-inkscape on CTAN:
%%   http://tug.ctan.org/tex-archive/info/svg-inkscape
%%
\begingroup%
  \makeatletter%
  \providecommand\color[2][]{%
    \errmessage{(Inkscape) Color is used for the text in Inkscape, but the package 'color.sty' is not loaded}%
    \renewcommand\color[2][]{}%
  }%
  \providecommand\transparent[1]{%
    \errmessage{(Inkscape) Transparency is used (non-zero) for the text in Inkscape, but the package 'transparent.sty' is not loaded}%
    \renewcommand\transparent[1]{}%
  }%
  \providecommand\rotatebox[2]{#2}%
  \newcommand*\fsize{\dimexpr\f@size pt\relax}%
  \newcommand*\lineheight[1]{\fontsize{\fsize}{#1\fsize}\selectfont}%
  \ifx\svgwidth\undefined%
    \setlength{\unitlength}{579.57660345bp}%
    \ifx\svgscale\undefined%
      \relax%
    \else%
      \setlength{\unitlength}{\unitlength * \real{\svgscale}}%
    \fi%
  \else%
    \setlength{\unitlength}{\svgwidth}%
  \fi%
  \global\let\svgwidth\undefined%
  \global\let\svgscale\undefined%
  \makeatother%
  \begin{picture}(1,0.60949685)%
    \lineheight{1}%
    \setlength\tabcolsep{0pt}%
    \put(0,0){\includegraphics[width=\unitlength,page=1]{fiber.pdf}}%
    \put(0.58227893,0.51268135){\makebox(0,0)[lt]{\lineheight{1.04999995}\smash{\begin{tabular}[t]{l}$\ell_o$\end{tabular}}}}%
    \put(0.92002064,0.41275027){\makebox(0,0)[lt]{\lineheight{1.04999995}\smash{\begin{tabular}[t]{l}$\ell_o^\perp$\end{tabular}}}}%
    \put(0,0){\includegraphics[width=\unitlength,page=2]{fiber.pdf}}%
    \put(0.75538653,0.22150119){\makebox(0,0)[t]{\lineheight{1.04999995}\smash{\begin{tabular}[t]{c}$\mathbb{C}P^1$\end{tabular}}}}%
  \end{picture}%
\endgroup%